\documentclass[reqno,a4paper]{amsart}
\usepackage{amssymb,setspace,graphicx}
\usepackage{ifpdf}
\ifpdf
 \usepackage[hyperindex,pagebackref]{hyperref}
\else
 \expandafter\ifx\csname dvipdfm\endcsname\relax
 \usepackage[hypertex,hyperindex,pagebackref]{hyperref}
 \else
 \usepackage[dvipdfm,hyperindex,pagebackref]{hyperref}
 \fi
\fi
\allowdisplaybreaks[4]
\numberwithin{equation}{section}
\theoremstyle{plain}
\newtheorem{thm}{Theorem}[section]
\newtheorem{lem}{Lemma}[section]
\newtheorem{cor}{Corollary}[section]
\theoremstyle{remark}

\DeclareMathOperator{\arcsinh}{arcsinh}

\begin{document}

\title[Some inequalities for Neuman-S\'andor mean]
{A unified proof of inequalities and some new inequalities involving Neuman-S\'andor mean}

\author[W.-H. Li]{Wen-Hui Li}
\address[Li]{Department of Mathematics, School of Science,
Tianjin Polytechnic University, Tianjin City, 300387, China}
\email{\href{mailto: W.-H. Li <wen.hui.li102@gmail.com>}{wen.hui.li102@gmail.com},
\href{mailto: W.-H. Li <wen.hui.li@foxmail.com>}{wen.hui.li@foxmail.com}}

\author[F. Qi]{Feng Qi}
\address[Qi]{Department of Mathematics, School of Science,
 Tianjin Polytechnic University, Tianjin City, 300387, China}
\email{\href{mailto: F. Qi <qifeng618@gmail.com>}{qifeng618@gmail.com},
 \href{mailto: F. Qi <qifeng618@hotmail.com>}{qifeng618@hotmail.com},
 \href{mailto: F. Qi <qifeng618@qq.com>}{qifeng618@qq.com}}
\urladdr{\url{http://qifeng618.wordpress.com}}



\begin{abstract}
In the paper, by finding linear relations of differences between some means, the authors supply a unified proof of some double inequalities for bounding Neuman-S\'andor means in terms of the arithmetic, harmonic, and contra-harmonic means and discover some new sharp inequalities involving Neuman-S\'andor, contra-harmonic, root-square, and other means of two positive real numbers.
\end{abstract}

\maketitle
\section{Introduction}
It is well known that the quantities
\begin{align*}
A(a,b)&=\frac{a+b}2,&G(a,b)&=\sqrt{ab}\,,\\
H(a,b)&=\frac{2ab}{a+b},&\overline{C}(a,b)&=\frac{2(a^2+ab+b^2)}{3(a+b)},\\
C(a,b)&=\frac{a^2+b^2}{a+b},&P(a,b)&=\frac{a-b}{4\arctan\sqrt{a/b}\,-\pi},\\
Q(a,b)&=\sqrt{\frac{a^2+b^2}{2}}\,,&T(a,b)&=\frac{a-b}{2\arctan\frac{a-b}{a+b}}
\end{align*}
are respectively called in the literature the arithmetic, geometric, harmonic, centroidal,
contra-harmonic, first Seiffert, root-square, and second Seiffert means of two positive real
numbers $a$ and $b$ with $a\ne b$.
\par
For $a,b>0$ with $a\ne b$, Neuman-S\'andor mean $M(a,b)$ is defined in~\cite{Neuman-M-Def} by
\begin{equation*}
M(a,b)=\frac{a-b}{2\arcsinh\frac{a-b}{a+b}},
\end{equation*}
where $\arcsinh x=\ln(x+\sqrt{x^2+1}\,)$ is the inverse hyperbolic sine function. At the same time, a chain of inequalities
\begin{equation*}
G(a,b)<L_{-1}(a,b)<P(a,b)<A(a,b)<M(a,b)<T(a,b)<Q(a,b)
\end{equation*}
were given in~\cite{Neuman-M-Def}, where
\begin{align*}
L_p(a,b)&=
\begin{cases}\displaystyle
\biggl[\frac{b^{p+1}-a^{p+1}}{(p+1)(b-a)}\biggr]^{1/p}, & p\ne-1,0 \\ \displaystyle
\frac1e\biggl(\frac{b^b}{a^a}\biggr)^{1/(b-a)}, & p=0\\ \displaystyle
\frac{b-a}{\ln b-\ln a},& p=-1
\end{cases}
\end{align*}
is the $p$-th generalized logarithmic mean of $a$ and $b$ with $a\ne b$.
\par
In~\cite{Neuman-M-Def, Neuman-M-Def-2}, it was established that
\begin{gather*}
A(a,b)<M(a,b)<T(a,b),\quad
P(a,b)<M(a,b)<T^2(a,b),\\
A(a,b)T(a,b)<M^2(a,b)<\frac{A^2(a,b)+T^2(a,b)}2
\end{gather*}
for $a,b>0$ with $a\ne b$.\par
For $0<a,b<\frac12$ with $a\ne b$, Ky Fan type inequalities
\begin{multline*}
\frac{G(a,b)}{G(1-a,1-b)}<\frac{L_{-1}(a,b)}{L_{-1}(1-a,1-b)}<\frac{P(a,b)}{P(1-a,1-b)}\\
<\frac{A(a,b)}{A(1-a,1-b)}
<\frac{M(a,b)}{M(1-a,1-b)}<\frac{T(a,b)}{T(1-a,1-b)}
\end{multline*}
were presented in~\cite[Proposition~2.2]{Neuman-M-Def}.
\par
In~\cite{L-L-C-NS}, it was showed that the double inequality
\begin{equation*}
L_{p_0}(a,b)<M(a,b)<L_2(a,b)
\end{equation*}
holds for all $a,b>0$ with $a\ne b$ and for $p_0=1.843\dotsc$, where $p_0$ is the unique solution of the
equation $(p+1)^{1/p}=2\ln(1+\sqrt2\,)$.\par
In~\cite{Neuman-note}, Neuman proved that the double inequalities
\begin{equation*}
\alpha Q(a,b)+(1-\alpha)A(a,b)<M(a,b)<\beta Q(a,b)+(1-\beta)A(a,b)
\end{equation*}
and
\begin{equation}\label{Neuman-eq2}
\lambda C(a,b)+(1-\lambda)A(a,b)<M(a,b)<\mu C(a,b)+(1-\mu)A(a,b)
\end{equation}
hold for all $a,b>0$ with $a\ne b$ if and only if
\begin{equation*}
\alpha\le\frac{1-\ln(1+\sqrt2\,)}{(\sqrt2-1)\ln(1+\sqrt2\,)}=0.3249\dotsc,\quad
\beta\ge\frac13
\end{equation*}
and
\begin{equation*}
\lambda\le\frac{1-\ln(1+\sqrt2\,)}{\ln(1+\sqrt2\,)}=0.1345\dotsc,\quad
\mu\ge\frac16.
\end{equation*}
\par
In~\cite[Theorems~1.1 to~1.3]{Z-C-L-optimal}, it was found that the double inequalities
\begin{equation*}
\alpha_1 H(a,b)+(1-\alpha_1)Q(a,b)<M(a,b)<\beta_1 H(a,b)+(1-\beta_1)Q(a,b),
\end{equation*}
\begin{equation*}
\alpha_2 G(a,b)+(1-\alpha_2)Q(a,b)<M(a,b)<\beta_2 G(a,b)+(1-\beta_2)Q(a,b),
\end{equation*}
and
\begin{equation}\label{zhao-eq1}
\alpha_3 H(a,b)+(1-\alpha_3)C(a,b)<M(a,b)<\beta_3 H(a,b)+(1-\beta_3)C(a,b)
\end{equation}
hold for all $a,b>0$ with $a\ne b$ if and only if
\begin{gather*}
\alpha_1\ge\frac29=0.2222\dotsc,\quad
\beta_1\le1-\frac1{\sqrt2\,\ln(1+\sqrt2\,)}=0.1977\dotsc, \\
\alpha_2\ge\frac13=0.3333\dotsc,\quad
\beta_2\le1-\frac1{\sqrt2\,\ln(1+\sqrt2\,)}=0.1977\dotsc,\\
\alpha_3\ge1-\frac1{2\ln(1+\sqrt2\,)}=0.4327\dotsc, \quad\beta_3\le\frac5{12}=0.4166\dotsc.
\end{gather*}
\par
In~\cite[Theorem~3.1]{Z-C-identric}, it was established that the double inequality
\begin{equation*}
\alpha I(a,b)+(1-\alpha)Q(a,b)<M(a,b)<\beta I(a,b)+(1-\beta)Q(a,b)
\end{equation*}
holds for all $a,b>0$ with $a\ne b$ if and only if
\begin{equation*}
\alpha\ge\frac12\quad \text{and} \quad
\beta\le\frac{e\bigl[\sqrt2\,\ln(1+\sqrt2\,)-1\bigr]}
{(\sqrt2\,e-2)\ln(1+\sqrt2\,)}=0.4121\dotsc.
\end{equation*}
\par
For more information on this topic, please refer to~\cite{chu-AAA, C-L-G-S-logarithmic, C-W-refinements, J-Q-contra, Jiang-Neuman, L-L-C-NS, Toader-M-Li-Zheng-Rev.tex, Neuman-note, Neuman-SB-mean, N-J-compan, Neuman-M-Def-2, W-C-L-Sharp, Zhao-note, Z-C-L-optimal, Zhao-geometric} and plenty of references cited therein.
\par
The first goal of this paper is, by finding linear relations of differences between some means, to supply a unified proof of inequalities~\eqref{Neuman-eq2} and~\eqref{zhao-eq1}.
\par
The second purpose of this paper is to establish some new sharp inequalities involving Neuman-S\'andor,
centroidal, contra-harmonic, and root-square means of two positive real numbers $a$ and $b$ with $a\ne b$.

\section{Lemmas}
In order to attain our aims, the following lemmas are needed.

\begin{lem}\cite[Lemma~1.1]{S-M-zero-balanced}\label{lem2.1}
Suppose that the power series $f(x)=\sum_{n=0}^{\infty}a_nx^n$ and $g(x)=\sum_{n=0}^{\infty}b_nx^n$
have the radius of convergence $r>0$ and $b_n>0$ for all $n\in\mathbb{N}=\{0,1,2,\dotsc\}$.
Let $h(x)=\frac{f(x)}{g(x)}$. Then the following statements are true.
\begin{enumerate}
\item
If the sequence $\{\frac{a_n}{b_n}\}_{n=0}^{\infty}$ is \textup{(}strictly\textup{)} increasing
\textup{(}decreasing\textup{)}, then $h(x)$ is also \textup{(}strictly\textup{)}
increasing \textup{(}decreasing\textup{)} on $(0,r)$.
\item
If the sequence $\{\frac{a_n}{b_n}\}$ is \textup{(}strictly\textup{)} increasing \textup{(}decreasing\textup{)}
for $0<n\le n_0$ and \textup{(}strictly\textup{)}
decreasing \textup{(}increasing\textup{)} for $n>n_0$, then there exists $x_0\in(0,r)$ such
that $h(x)$ is \textup{(}strictly\textup{)}
increasing \textup{(}decreasing\textup{)} on $(0,x_0)$ and \textup{(}strictly\textup{)} decreasing
\textup{(}increasing\textup{)} on $(x_0,r)$.
\end{enumerate}
\end{lem}

\begin{lem}\label{lem2.2-h1}
Let
\begin{equation}
h_1(x)=\frac{\sinh x-x}{2x\sinh^2x}.
\end{equation}
Then $h_1(x)$ is strictly decreasing on $(0,\infty)$ and has the limit $\lim_{x\to0^+}h_1(x)=\frac1{12}$.
\end{lem}

\begin{proof}
Let
$f_1(x)=\sinh x-x$ and $f_2(x)=2x\sinh^2x=x\cosh2x-x$.
Using the power series
\begin{equation}\label{sinh-cosh-series}
\sinh x=\sum_{n=0}^{\infty}\frac{x^{2n+1}}{(2n+1)!} \quad\text{and}\quad
\cosh x=\sum_{n=0}^{\infty}\frac{x^{2n}}{(2n)!},
\end{equation}
we can express the function $f_1(x)$ and $f_2(x)$ as
\begin{equation}\label{h1-f1=eq1}
f_1(x)=\sum_{n=0}^{\infty}\frac{x^{2n+3}}{(2n+3)!}\quad
\text{and}\quad
f_2(x)=\sum_{n=0}^{\infty}\frac{2^{2n+2}x^{2n+3}}{(2n+2)!}.
\end{equation}
Hence, we have
\begin{equation}\label{h1-f1-f2-eq}
h_1(x)=\frac{\sum_{n=0}^{\infty}a_nx^{2n}}{\sum_{n=0}^{\infty}b_nx^{2n}},
\end{equation}
where $a_n=\frac1{(2n+3)!}$ and $b_n=\frac{2^{2n+2}}{(2n+2)!}$.
Let $c_n=\frac{a_n}{b_n}$. Then $c_n=\frac1{(2n+3)2^{2n+2}}$ and
\begin{equation*}
c_{n+1}-c_n=\frac{-(6n+17)}{(2n+3)(2n+5)2^{2n+4}}<0.
\end{equation*}
As a result, by Lemma~\ref{lem2.1}, it follows that the function $h_1(x)$ is strictly decreasing
on $(0,\infty)$.
\par
From~\eqref{h1-f1-f2-eq}, it is easy to see that $\lim_{x\to0^+}h_1(x)=\frac{a_0}{b_0}=\frac1{12}$.
The proof of Lemma~\ref{lem2.2-h1} is complete.
\end{proof}

\begin{lem}\label{lem2.3-h2}
Let
\begin{equation}
h_2(x)=\frac{1-\frac{\sinh x}x+\frac{\sinh^2x}3}{\cosh x-\frac{\sinh x}x}.
\end{equation}
Then $h_2(x)$ is strictly increasing on $(0,\infty)$ and has the limit $\lim_{x\to0^+}h_2(x)=\frac12$.
\end{lem}

\begin{proof}
Let
\begin{equation*}
f_3(x)=1-\frac{\sinh x}x+\frac{\sinh^2x}3=1-\frac{\sinh x}x+\frac{\cosh2x-1}6
\end{equation*}
and
\begin{equation*}
f_4(x)=\cosh x-\frac{\sinh x}x.
\end{equation*}
Making use of the power series in~\eqref{sinh-cosh-series} shows that
\begin{equation*}
f_3(x)=\sum_{n=0}^{\infty}\frac{(2n+3)2^{2n+2}-6}{6(2n+3)!}x^{2n+2}\quad
\text{and}\quad
f_4(x)=\sum_{n=0}^{\infty}\frac{2n+2}{(2n+3)!}x^{2n+2}.
\end{equation*}
Therefore, we have
\begin{equation}\label{lem2.3-h2-f3-f4-eq}
h_2(x)=\frac{\sum_{n=0}^{\infty}a_nx^{2n+2}}
{\sum_{n=0}^{\infty}b_nx^{2n+2}},
\end{equation}
where $a_n=\frac{(2n+3)2^{2n+2}-6}{6(2n+3)!}$ and $b_n=\frac{2n+2}{(2n+3)!}$.
Let $c_n=\frac{a_n}{b_n}$. Then
\begin{equation*}
c_n=\frac{(2n+3)2^{2n+1}-3}{6(n+1)}
\end{equation*}
and
\begin{equation*}
c_{n+1}-c_n=\frac{3+7\cdot2^{2n+2}+21n\cdot2^{2n+1}+3n^2\cdot2^{2n+2}}{6(n+1)(n+2)}>0.
\end{equation*}
Accordingly, by Lemma~\ref{lem2.1}, it follows that the function $h_2(x)$ is strictly increasing on $(0,\infty)$.
\par
It is clear that $\lim_{x\to0^+}h_2(x)=\frac{a_0}{b_0}=\frac12$.
The proof of Lemma~\ref{lem2.3-h2} is complete.
\end{proof}

\begin{lem}\label{lem2.4-h3}
Let
\begin{equation}
h_3(x)=\frac{\cosh x-\frac{\sinh x}x}{1+\sinh^2x-\frac{\sinh x}x}.
\end{equation}
Then $h_3(x)$ is strictly decreasing on $(0,\infty)$ and has the limit $\lim_{x\to0^+}h_3(x)=\frac25$.
\end{lem}

\begin{proof}
Let
\begin{equation*}
f_5(x)=\cosh x-\frac{\sinh x}x
\end{equation*}
and
\begin{equation*}
f_6(x)=1+\sinh^2x-\frac{\sinh x}x=1-\frac{\sinh x}x+\frac{\cosh2x-1}2.
\end{equation*}
Utilizing the power series in~\eqref{sinh-cosh-series} gives
\begin{equation*}
f_5(x)=\sum_{n=0}^{\infty}\frac{2n+2}{(2n+3)!}x^{2n+2}\quad\text{and}\quad
f_6(x)=\sum_{n=0}^{\infty}\frac{(2n+3)2^{2n+1}-1}{(2n+3)!}x^{2n+2}.
\end{equation*}
This implies that
\begin{equation}\label{lem2.4-h3-f5-f6-eq}
h_3(x)=\frac{\sum_{n=0}^{\infty}a_nx^{2n+2}}
{\sum_{n=0}^{\infty}b_nx^{2n+2}},
\end{equation}
where $a_n=\frac{2n+2}{(2n+3)!}$ and $b_n=\frac{(2n+3)2^{2n+1}-1}{(2n+3)!}$.
Let $c_n=\frac{a_n}{b_n}$. Then
\begin{equation*}
c_n=\frac{2n+2}{(2n+3)2^{2n+1}-1}
\end{equation*}
and
\begin{equation*}
c_{n+1}-c_n=-\frac{2(1+7\cdot2^{2n+2}+21n\cdot2^{2n+1}+3n^2\cdot2^{2n+2})}
{(3\cdot2^{2n+1}+n\cdot2^{2n+2}-1)(5\cdot2^{2n+3}+n\cdot2^{2n+4}-1)}<0.
\end{equation*}
In light of Lemma~\ref{lem2.1}, we obtain that the function $h_3(x)$ is strictly decreasing on $(0,\infty)$.
\par
It is obvious that $\lim_{x\to0^+}h_3(x)=\frac{a_0}{b_0}=\frac25$.
The proof of Lemma~\ref{lem2.4-h3} is complete.
\end{proof}

\section{A unified proof of inequalities~\eqref{Neuman-eq2} and~\eqref{zhao-eq1}}

Now we are in a position to supply a unified proof of inequalities~\eqref{Neuman-eq2} and~\eqref{zhao-eq1} and, as corollaries, to establish some new inequalities involving Neuman-S\'andor, contra-harmonic, centroidal, and root-square means of two positive real numbers $a$ and $b$ with $a\ne b$.
\par
It is not difficult to see that the inequalities~\eqref{Neuman-eq2} and~\eqref{zhao-eq1} can
be rearranged respectively as
\begin{equation}\label{Neuman-eq2-1}
\lambda-1<\frac{M(a,b)-C(a,b)}{C(a,b)-A(a,b)}<\mu-1
\end{equation}
and
\begin{equation}\label{zhao-eq1-1}
-\alpha_3<\frac{M(a,b)-C(a,b)}{C(a,b)-H(a,b)}<-\beta_3.
\end{equation}
The denominators in~\eqref{Neuman-eq2-1} and~\eqref{zhao-eq1-1} meet
\begin{equation}\label{lines-CH}
2[C(a,b)-A(a,b)]=C(a,b)-H(a,b)=\frac{(a-b)^2}{a+b}
\end{equation}
which were presented in~\cite[Eq.~(4.4)]{Jiang-serffert}.
This implies that the inequalities~\eqref{Neuman-eq2} and~\eqref{zhao-eq1} are identical up to a
scalar. Therefore, it is sufficient to prove one of the two inequalities~\eqref{Neuman-eq2}
and~\eqref{zhao-eq1}.
\par
By a direct calculation, we also find
\begin{multline}\label{line-equation}
6[\overline{C}(a,b)-A(a,b)]=3[C(a,b)-\overline{C}(a,b)] =2[A(a,b)-H(a,b)]\\
=\frac32[\overline{C}(a,b)-H(a,b)]=\frac{(a-b)^2}{a+b}\triangleq CH(a,b).
\end{multline}
So, it is natural to raise a problem: what is the best constants $\alpha$ and $\beta$ such that the double inequality
\begin{equation}\label{thm3.1-eq}
\alpha<\frac{M(a,b)-C(a,b)}{CH(a,b)}<\beta
\end{equation}
holds for all $a,b>0$ with $a\ne b$? The following theorem gives a solution to this problem.

\begin{thm}\label{thm3.1}
The double inequality~\eqref{thm3.1-eq} holds for all $a,b>0$ with $a\ne b$ if and only if
\begin{equation*}
\alpha\le\frac1{2\ln(1+\sqrt2\,)}-1=-0.4327\dotsc\quad\text{and}\quad\beta\ge-\frac5{12}=-0.4166\dotsc.
\end{equation*}
\end{thm}

\begin{proof}
Without loss of generality, we assume that $a>b>0$. Let $x=\frac{a}b$. Then $x>1$ and
\begin{equation*}
\frac{M(a,b)-C(a,b)}{CH(a,b)}=\frac{\frac{x-1}{2\arcsinh{\frac{x-1}{x+1}}}
-\frac{x^2+1}{x+1}}{\frac{(x-1)^2}{x+1}}.
\end{equation*}
Let $t=\frac{x-1}{x+1}$. Then $t\in(0,1)$ and
\begin{equation*}
\frac{M(a,b)-C(a,b)}{CH(a,b)}=\frac{\frac{t}{\arcsinh t}
-t^2-1}{2t^2}.
\end{equation*}
Let $t=\sinh\theta$ for $\theta\in\bigl(0,\ln\bigl(1+\sqrt2\,\bigr)\bigr)$. Then
\begin{equation*}
\frac{M(a,b)-C(a,b)}{CH(a,b)}=\frac{\frac{\sinh\theta}{\theta}
-\sinh^2\theta-1}{2\sinh^2\theta}
=\frac{\sinh\theta-\theta}{2\theta\sinh^2\theta}-\frac12.
\end{equation*}
In virtue of Lemma~\ref{lem2.2-h1}, Theorem~\ref{thm3.1} is thus proved.
\end{proof}

\begin{cor}\label{cor3.1}
The double inequality
\begin{equation}\label{cor3.1-eq1}
\alpha CH(a,b)+M(a,b)< C(a,b)<\beta CH(a,b)+M(a,b)
\end{equation}
holds for all $a,b>0$ with $a\ne b$ if and only if
$\alpha\le\frac5{12}=0.4166\dotsc$ and
\begin{equation*}
\beta\ge1-\frac1{2\ln(1+\sqrt2\,)}=0.4327\dotsc.
\end{equation*}
\end{cor}

\begin{cor}
The double inequality
\begin{equation}
\alpha CH(a,b)+M(a,b)< \overline{C}(a,b)<\beta CH(a,b)+M(a,b)
\end{equation}
holds for all $a,b>0$ with $a\ne b$ if and only if $\alpha\le\frac1{12}=0.0833\dotsc$
and
\begin{equation*}
\beta\ge\frac23-\frac1{2\ln(1+\sqrt2\,)}=0.0993\dotsc.
\end{equation*}
\end{cor}

\section{Some new inequalities involving Neuman-S\'andor mean}

Finally we further establish some new inequalities involving Neuman-S\'andor, centroidal, root-square, and other means.

\begin{thm}\label{thm3.2}
The inequality
\begin{equation}
M(a,b)>\lambda CH(a,b)
\end{equation}
holds for all $a,b>0$ with $a\ne b$ if and only if $\lambda\le\frac1{2\ln(1+\sqrt2\,)}=0.5672\dotsc$.
\end{thm}

\begin{proof}
It is clear that
\begin{equation*}
\frac{M(a,b)}{CH(a,b)}=\frac{(a-b)(a+b)}{(a-b)^22\arcsinh\frac{a-b}{a+b}}
=\frac{a+b}{a-b}\frac1{2\arcsinh\frac{a-b}{a+b}}.
\end{equation*}
Without loss of generality, we assume that $a>b>0$. Let $x=\frac{a-b}{a+b}$. Then $x\in(0,1)$ and
\begin{equation*}
\frac{M(a,b)}{CH(a,b)}=\frac1{2x\arcsinh x}\triangleq f(x).
\end{equation*}
Differentiating $f(x)$ yields
\begin{equation*}
f'(x)=-\frac{\frac{x}{\sqrt{1+x^2}}\,+\arcsinh x}{2x^2\arcsinh^2 x}\le0
\end{equation*}
which means that function $f(x)$ is decreasing for $x\in(0,1)$.
\par
It is apparent that
\begin{equation*}
\lim_{x\to1^-}f(x)=\frac1{2\ln(1+\sqrt2\,)}.
\end{equation*}
The proof of Theorem~\ref{thm3.2} is thus complete.
\end{proof}

\begin{thm}\label{thm3.3}
The double inequality
\begin{equation}
\alpha Q(a,b)+(1-\alpha)M(a,b)<\overline{C}(a,b)<\beta Q(a,b)+(1-\beta)M(a,b)
\end{equation}
holds for all $a,b>0$ with $a\ne b$ if and only if $\alpha\le\frac12$ and
\begin{equation*}
\beta\ge\frac{3-4\ln(1+\sqrt2\,)}{3[1-\sqrt2\,\ln(1+\sqrt2\,)]}=0.7107\dotsc.
\end{equation*}
\end{thm}

\begin{proof}
It is sufficient to show
\begin{equation*}
\alpha<\frac{\overline{C}(a,b)-M(a,b)}{Q(a,b)-M(a,b)}<\beta.
\end{equation*}
Without loss of generality, we assume that $a>b>0$. Let $x=\frac{a}b$. Then $x>1$ and
\begin{equation*}
\frac{\overline{C}(a,b)-M(a,b)}{Q(a,b)-M(a,b)}
=\frac{\frac{2(x^2+x+1)}{3(x+1)}-\frac{x-1}{2\arcsinh{\frac{x-1}{x+1}}}
}{\sqrt{\frac{x^2+1}2}\,-\frac{x-1}{2\arcsinh{\frac{x-1}{x+1}}}}.
\end{equation*}
Let $t=\frac{x-1}{x+1}$. Then $t\in(0,1)$ and
\begin{equation*}
\frac{\overline{C}(a,b)-M(a,b)}{Q(a,b)-M(a,b)}
=\frac{\frac{t^2}3+1-\frac{t}{\arcsinh{t}}}{\sqrt{1+t^2}\,-\frac{t}{\arcsinh{t}}}.
\end{equation*}
Let $t=\sinh\theta $ for $\theta\in\bigl(0,\ln\bigl(1+\sqrt2\,\bigr)\bigr)$. Then
\begin{equation*}
\frac{\overline{C}(a,b)-M(a,b)}{Q(a,b)-M(a,b)}
=\frac{\frac{\sinh^2\theta}3+1-\frac{\sinh\theta}\theta}{\cosh\theta-\frac{\sinh\theta}\theta}.
\end{equation*}
By Lemma~\ref{lem2.3-h2}, we obtain Theorem~\ref{thm3.3}.
\end{proof}

\begin{thm}\label{thm3.4}
The double inequality
\begin{equation}\label{thm3.4-eq1}
\alpha C(a,b)+(1-\alpha)M(a,b)<Q(a,b)<\beta C(a,b)+(1-\beta)M(a,b)
\end{equation}
holds for all $a,b>0$ with $a\ne b$ if and only if
\begin{equation*}
\alpha\le\frac{\sqrt2\,\ln(1+\sqrt2\,)-1}{2\ln(1+\sqrt2\,)-1}=0.3231\dotsc
\quad \text{and}\quad \beta\ge\frac25.
\end{equation*}
\end{thm}

\begin{proof}
The double inequalities~\eqref{thm3.4-eq1} is the same as
\begin{equation*}
\alpha<\frac{Q(a,b)-M(a,b)}{C(a,b)-M(a,b)}<\beta.
\end{equation*}
Without loss of generality, we assume that $a>b>0$. Let $x=\frac{a}b$. Then $x>1$ and
\begin{equation*}
\frac{Q(a,b)-M(a,b)}{C(a,b)-M(a,b)}=\frac{\sqrt{\frac{x^2+1}2}\,-\frac{x-1}{2\arcsinh{\frac{x-1}{x+1}}}
}{\frac{x^2+1}{x+1}-\frac{x-1}{2\arcsinh{\frac{x-1}{x+1}}}}.
\end{equation*}
Let $t=\frac{x-1}{x+1}$. Then $t\in(0,1)$ and
\begin{equation*}
\frac{Q(a,b)-M(a,b)}{C(a,b)-M(a,b)}=\frac{\sqrt{1+t^2}\,-\frac{t}{\arcsinh{t}}}{1+t^2-\frac{t}{\arcsinh{t}}}.
\end{equation*}
Let $t=\sinh\theta $ for $\theta\in\bigl(0,\ln\bigl(1+\sqrt2\,\bigr)\bigr)$. Then
\begin{equation*}
\frac{Q(a,b)-M(a,b)}{C(a,b)-M(a,b)}
=\frac{\cosh\theta-\frac{\sinh\theta}\theta}{1+\sinh^2\theta-\frac{\sinh\theta}\theta}.
\end{equation*}
According to Lemma~\ref{lem2.4-h3}, the proof of Theorem~\ref{thm3.3} is complete.
\end{proof}

\end{document}